\documentclass[12pt]{amsart}

\usepackage{amsmath}
\usepackage{amssymb}
\usepackage{amsthm}
\usepackage[utf8]{inputenc}
\usepackage{booktabs}
\usepackage{xcolor}
\usepackage{graphicx}
\usepackage{subcaption}
\usepackage{float}
\usepackage{nicefrac}
\usepackage{tikz}
\usetikzlibrary{positioning}
\usepackage[percent]{overpic}
\usepackage{mathtools}
\usepackage{hyperref}

\makeatletter
\def\subsection{
	\@startsection{subsection}{2}\z@{.5\linespacing\@plus.7\linespacing}{.5\baselineskip}{\bfseries\centering}
}
\makeatother

\theoremstyle{definition}
\newtheorem{Def}{Definition}[section]

\newtheorem*{Not}{Note}
\newtheorem{Exe}{Example}

\newtheorem{Rem}{Remark}[section]

\theoremstyle{theorem}
\newtheorem{The}{Theorem}[section]
\newtheorem{Pro}{Proposition}[section]
\newtheorem{Cor}{Corollary}[section]
\newtheorem{Lem}{Lemma}[section]

\newcommand{\R}{\mathbb{R}}

\newcommand{\N}{\mathbb{N}}
\newcommand{\Z}{\mathbb{Z}}

\newcommand{\lcs}{\frak{lcs}}

\title{Morse-Novikov homology and $\beta$-critical points}

\author{Adrien Currier}
\address{Nantes Université, Laboratoire de Mathématiques Jean Leray, LMJL,
	UMR 6629, F-44000 Nantes, France}
\email{arien.currier@univ-nantes.fr}
\date{}

\begin{document}
	\maketitle
	\begin{abstract} Given a manifold $M$, some closed $\beta\in\Omega^1(M)$ and a map $f\in C^\infty(M)$, a $\beta$-critical point is some $x\in M$ such that $d_\beta f_{x}=0$ for the Lichnerowicz derivative $d_\beta$. In this paper, we will give a lower bound for the number of $\beta$-critical points of index $i$ of a $\beta$-Morse function $f$ in terms of the Morse-Novikov homology, and we generalize this result to generating functions (quadratic at infinity). We also give an application to the detection of essential Liouville chords of a set length. These are a type of chords that appear in locally conformally symplectic ($\lcs$) geometry as even-dimensional analogues to Reeb chords.
	\end{abstract}
\tableofcontents
\section{Introduction}
Morse-Novikov homology (see second section for an introduction to Morse-Novikov homology) was first introduced by S. Novikov in the 80s (see \cite{Novikov1981MULTIVALUEDFA} and \cite{Novikov1982TheHF}). Note that in the literature, what appears as Morse-Novikov cohomology is often not the cohomology theory stemming from the Morse-Novikov chain complex, but the cohomology of the chain complex $(\Omega^*(M),d_\beta)$ on a manifold $M$, with $d_\beta$ the Lichnerowicz derivative for some closed $\beta\in\Omega^1(M)$.
\begin{Def}
	Let $M$ be a manifold and $\beta\in\Omega^1(M)$ be closed. The Lichnerowicz derivative associated to $\beta$ is the map:
\begin{align*}
	d_\beta:\Omega^*(M)&\rightarrow\Omega^{*+1}(M)\\
	\alpha&\mapsto d\alpha-\beta\wedge\alpha
\end{align*}
\end{Def}
Do note that this map is indeed linear and $d_\beta^2=0$.

To avoid any confusion, the cohomology of $(\Omega^*(M),d_\beta)$ will be called the Lichnerowicz cohomology. While it is indeed ``almost'' equal Morse-Novikov cohomology, it can differ in some cases.\newline

Since its first formulation, Morse-Novikov homology has appeared in multiple places, from the study of fixed points of symplectomorphisms (cf. \cite{Vn1995SymplecticFP}), to knot theory (cf. \cite{Weber2002MorseNovikovNF}). A notable use of this homology theory comes from Chantraine and Murphy's paper (\cite{Chantraine2016ConformalSG}) in which they show that the total number of $\beta$-critical points (see definition below) of a generic generating function is at least the sum of the ranks of the Morse-Novikov homology group. This result is to be compared with the results of classical Morse theory, for which the number of critical point of a specific index can be bounded from below not only by the rank of the Morse homology, but also by the torsion numbers.

\begin{Def}
Let $M$ be a manifold, $\beta\in\Omega^1(M)$ be closed and $f:M\rightarrow\R$ be smooth.
\begin{enumerate}
	\item The set of $\beta$-critical points of $f$ will be $Crit^\beta(f):=\{x\in M: d_\beta f_{|x}=0\}$.
	\item Let $\tilde{M}_\beta$ be the integral cover of $\beta$. Take $g$ a primitive of $\beta$ on that cover and $F$ the pullback of $f$ to the cover. Then $f$ will be said to be $\beta$-Morse if and only $e^{-g}F$ is Morse. Equivalently, $f$ is said to be $\beta$-Morse if $d_\beta f$ intersects the $0$-section transversely. 
	\item If $f$ is $\beta$-Morse, then a $\beta$-critical point $x$ will be said to be of index $j$ if there is a locally defined primitive $g$ of $\beta$ around that point such that $x$ is a critical point of index $j$ of the locally-defined Morse function $e^{-g}f$. The set of $\beta$-critical points of index $j$ will be written $Crit_j^\beta(f)$.
\end{enumerate}
\end{Def}

Note that being $\beta$-Morse is a generic property. Morse-Novikov homology is uniquely adapted to study the $\beta$-critical points; let us see why with an example. Since $\beta=d_\beta(-1)$, the first $\beta$-Morse function to consider is the constant function equal to $-1$ when $\beta$ intersects the $0$-section transversely.

A naive adaptation of Morse homology would lead us to consider a chain complex freely generated of the critical points, and whose homology count the flow lines of a ``gradient'' vector field $X$ that is the dual of $d_\beta (-1)$ for some Riemannian metric. However, in classical Morse theory, following those flow lines would lead the function to decrease, and, in particular, following those flow lines from a critical point to another will never let us backtrack to a previous critical point. Observe that this is not the case here: taking any genus $g\geq 2$ surface, a closed non-exact $1$-form on that surface will have critical points and its gradient vector field will have closed orbits. Therefore, it becomes necessary to keep track of how many loops we make through a local system. While several local coefficient homologies could keep track of such data, Morse-Novikov homology stands out by providing some of the best lower bounds for our toy example $d_\beta (-1)$, as long as it is $\beta$-Morse. Indeed, by classical Morse-Novikov theory, \[\# Crit^\beta_j(d_\beta (-1))\geq rank(HN_j(M,\beta))+q_j+q_{j-1},\] where $HN$ stands for the Morse-Novikov homology, and $q_j$ (resp. $q_{j-1}$) is the minimal number of generators of the torsion subgroup of the $j$-th (resp. $(j-1)$-th) Morse-Novikov homology group. Those two last quantities are known as the torsion number, while the rank of the $j$-th homology group is the $j$-th Novikov number. This bound is extremely similar to that given by Morse homology for Morse functions.\newline

This paper fits in the broader context of the study of rigidity phenomena in locally conformally symplectic ($\lcs$) geometry through generating functions and, and more precisely, is motivated by the need for a good ``Morse theory for Morse-Novikov homology''. Both of these are important stepping stones towards a sheaf theory for $\lcs$ geometry. Note that in \cite{currier2025projectionexactlagrangianslocally}, the author proves that not every exact Lagrangian (of ``$\lcs$ type'') has a generating function, putting some limitations on that strategy. However, other strategies, such as an adaptation of Floer homology, encounter many more hurdles (see \cite{Chantraine2016ConformalSG} for more on that). Nevertheless, some limited success has been achieved in adapting Floer homology (see \cite{Oh2023PCLCM}).\newline

One of the aims of this paper is to present a strategy for the proof of this theorem that can be directly translated to sheaf theory, allowing for a better study of $\lcs$ topology. Indeed, derived sheaf theory has been successfully employed by several authors to study symplectic geometry, such as Kashiwara and Shapira (\cite{Kashiwara1990SheavesOM}) or Guillermou (\cite{Guillermou2019SheavesAS}). As such, we will endeavor to prove:

\begin{The}\label{thm1}
Let $M$ be a closed connected manifold and $\beta\in\Omega^1(M)$ be closed. We will call $\beta$ the various pullbacks of $\beta$. Take $F:M\times\R^m\rightarrow\R$ (for some $m\in\mathbb{N}$) a smooth map that is equal to a quadratic form outside of a compact. Let $p$ be the dimension of the vector subspace on which $F$ is negative definite at infinity. Then, if $F$ is $\beta$-Morse, we have
\[\# Crit^\beta_j(F)\geq rank(HN_{j-p}(M,\beta))\]
where $HN$ stands for Morse-Novikov homology.\end{The}

As stated above, this is one step closer to Morse theory which state that (under some genericity conditions):\[\# Crit_j(F)\geq rank(H_{j-p}(M))+q_{j-p}'+q_{j-1-p}',\]
where $q_i'$ is the torsion number of $H_i(M)$ for each $i$. As this observation implies, there is still some work to do in that direction.\newline

Nevertheless, the strategies explored in section \ref{sec3} still have uses beyond proving the previous theorem and providing the beginning of a framework for a derived sheaf theory (similar to that of Guillermou, Kashiwara and Shapira) adapted to $\lcs$ geometry. To illustrate this, we will apply them to the study of (essential) Liouville chords.

First defined by L. Traynor in \cite{Traynor2001GeneratingFP}, the difference function (in contact geometry) can be used to count Reeb chords between two Legendrians. However, an adaptation of this function to $\lcs$ geometry requires some modifications as Liouville chords generically come in $1$-parameter families. To this end, we introduce the following modification:

\begin{Def}\label{t-diff}
	Let $F_1:M\times\mathbb{R}^{k_1}\rightarrow\R$ and $F_2:M\times\mathbb{R}^{k_2}\rightarrow\R$ be generating functions. Then the $t$-difference function is the map :
	\begin{align*}
		\Delta^t_{F_1,F_2}:M\times\R^{k_1}\times\R^{k_2}&\rightarrow \R\\
		(x,\xi_1,\xi_2)&\mapsto F_2(x,\xi_2)-e^tF_1(x,\xi_1)
	\end{align*}
\end{Def}

We can use this difference function to study the number of (essential) Liouville chords of length $t$ between two exact Lagrangians that have generating functions.

\begin{Pro}\label{prop}
	Assume that $F_1$ and $F_2$ are generating functions quadratic at infinity that are strictly positive on the $\beta$-exact Lagrangians they define, and that $\Delta^t_{F_1,F_2}$ is $\beta$-Morse.
	
	Then the number of Liouville chords of length $t$ from $L_{F_1}$ to $L_{F_2}$ is bigger or equal to $\sum_i rank(HN_i(M,\beta))$.
	
	If $0$ is a regular value of $\Delta^t_{F_1,F_2}$, then we also have the following assertions: 
	
	\begin{enumerate}
\item The number of essential (positive) Liouville chords of length $t\geq0$ from $L_{F_1}$ to $L_{F_2}$ is bigger or equal to: \[\sum_i rank(HN_i(\{\Delta^t_{F_1,F_2}\geq 0\},\beta)).\]
\item The number of essential negative Liouville chords of length $t\leq0$ from $L_{F_1}$ to $L_{F_2}$ is bigger or equal to: \[\sum_i rank(HN_i(\{\Delta^t_{F_1,F_2}\leq 0\},\beta)).\]
\item The number of essential Liouville chords of length $t>0$ from $L_{F_1}$ to itself is bigger or equal to: \[\sum_i rank(HN_i(\{\Delta^t_{F_1,F_1}\geq 0\},\beta)).\]
	\end{enumerate}
Here, the Morse-Novikov homologies are computed taking the pullback of $\beta$ to the pertinent sets.
\end{Pro}
\begin{Rem}
The first part of the proposition may be viewed as a generalization of the theorem \ref{thm1}. Indeed if the two Lagrangians intersect the $0$-section at the same point, they will have a Liouville chord of any length between them at that point. In the case where $F_1=F_2$, this means that each time $L_{F_1}$ intersects the $0$-section (that is to say, each time $d_\beta F_1=0$), there will be a Liouville chord of any length.
	
Also, note that when the two Lagrangians (of ``$\lcs$ type'') are different, length $0$ Liouville chords happen when the two Lagrangians intersect.
\end{Rem}
\begin{Rem}
A way to interpret essential Liouville chords is the following. Let $\Lambda_{e^tF_1}$ (resp. $\Lambda_{F_2}$) be the Legendrian lift of $L_{e^tF_1}$ (resp. $L_{F_2}$) to $J^1M$, endowed with its canonical contact form. An essential (positive) Liouville chord from $L_{F_1}$ to $L_{F_2}$ of length $t$ corresponds to a (positive) Reeb chord from $\Lambda_{e^tF_1}$ to $\Lambda_{F_2}$. The value of $\Delta^t_{F_1,F_2}$ the corresponds to the length of the Reeb chord.
\end{Rem}
This proposition ties in with the author's previous preprint (\cite{currier2025projectionexactlagrangianslocally}), in which it appears that Liouville chords are to $\lcs$ geometry what Reeb chords are to contact geometry. This observation motivates the adaptation and use of techniques that are used to study Reeb chords.\newline

The layout of this paper is as follows. We will start by giving reminders about Morse-Novikov homology in section \ref{sec2}, as well as explaining how it relates to Lichnerowicz cohomology to set the record straight. This will then be followed by a proof of this theorem for $m=0$ in section \ref{sec3}. Then, section \ref{sec4} will generalize the statement to $m\geq 0$, thus proving the above theorem. Finally, section \ref{sec5} will begin with a short reminder about $\lcs$ geometry and end with a proof of proposition \ref{prop}.\newline

\paragraph{\textbf{Acknowledgments.}} The author would like to thank both B. Chantraine and F. Laudenbach for their feedback.

\section{Morse-Novikov homology}\label{sec2}

Given a manifold $M$ and a closed $1$-form defined on $M$, there are two main ways of defining the Morse-Novikov homology of $M$ associated to $\beta$. The first, original, description of this homology is as the Morse homology associated to $\beta$.  With this description, the chain complex is freely generated over the critical points of $\beta$, and the differential counts the orbits (of the flow associated to a vector field which is the dual of $\beta$ for some metric) from one critical point to another. However the definition of the differential needs to be slightly modified to account for the potential presence of homoclinic orbits. The second definition of the Morse-Novikov homology, and the one which we will use in the rest of this paper, is more algebraic in nature.

The description of Morse-Novikov homology will be based off of Farber's book, \cite{Farber2004TopologyOC}, which is a good reference for all of one's Morse-Novikov-related needs. Do note that Farber's description of Morse-Novikov homology is somewhat different to that of Novikov. However, both homologies are equal whenever $\beta$ satisfies the conditions necessary to make Novikov's description work.

\begin{Def}
	Let $\Gamma$ be a subgroup of $(\mathbb{R},+)$. The Novikov ring is the ring: \[Nov(\Gamma)= \left\{ \sum_{i=0}^{+\infty}a_i t^{\gamma_i}\,:\,\forall i\in\mathbb{N}, a_i\in\mathbb{Z}, \gamma_i\in\Gamma\textit{ and }\lim_{+\infty}\gamma_i=-\infty \right\}.\]
	Let $\sum_{\alpha\in\Gamma}a_\alpha t^\alpha$ and $\sum_{\beta\in\Gamma}b_\beta t^\beta$ be two elements of $Nov(\Gamma)$, then the sum is given by:
	\[\sum_{\alpha\in\Gamma}a_\alpha t^\alpha+\sum_{\beta\in\Gamma}b_\beta t^\beta=\sum_{\alpha\in\Gamma}(a_\alpha+b_\alpha) t^\alpha;\]
	and the product is given by:
	\[\sum_{\alpha\in\Gamma}a_\alpha t^\alpha\times\sum_{\beta\in\Gamma}b_\beta t^\beta=\sum_{\alpha,\beta\in\Gamma}(a_{\alpha}\times b_{\beta}) t^{\alpha+\beta};\]
\end{Def}
Through this ring, we can define a local coefficients system as follows:
\begin{Def}
	Let $M$ be a manifold and $\beta\in\Omega^1(M)$ be closed. Let \[\Gamma_\beta:=\{<[\beta],h(\gamma)>\,:\,\gamma\in\pi_1(M)\}\] where $[\beta]$ is the cohomology class of $\beta$ and $h$ is the Hurewicz morphism. The morphism of $\Z$-modules: \begin{align*}
		\phi\,:\, \mathbb{Z}[\pi_1(M)]\rightarrow & Nov(\Gamma_\beta)\\
		\gamma\in\pi_1(M)\mapsto & t^{<[\beta],h(\gamma)>}
	\end{align*} defines a local coefficient system $\mathcal{L}_\beta$. 
	
	The Morse-Novikov chain complex is the chain complex $C_*(M,\mathcal{L}_\beta)$. The homology of this complex is the Morse-Novikov homology $H_*(M,\mathcal{L}_\beta)$, denoted by $HN_*(M,\beta)$.
\end{Def}

\begin{Rem}
	Some define the Morse-Novikov homology by taking the homology after tensorizing the chain complex $C_*(M,\mathcal{L}_\beta)$ with $Nov(\mathbb{R})$, which is a flat $Nov(\Gamma_\beta)$-module. However, applying the universal coefficient theorem (and keeping in mind the flatness of $Nov(\Gamma_\beta)$) shows that the difference is merely cosmetic.
\end{Rem}

We should point out that the definition of $\Gamma_\beta$ does not depend on the choice of an element in the cohomology class of $\beta$. This leads us to the following:

\begin{Pro}[\cite{Farber2004TopologyOC}]
	Let $\beta$ and $\beta'$ be two closed $1$-forms on a manifold $M$ such that $[\beta]=[\beta']\in H^1_{dR}(M)$, then \[HN_*(M,\beta)\simeq HN_*(M,\beta').\] In particular, for any $g\in C^\infty(M)$, then \[HN_*(M,dg)=H_*(M,\mathbb{Z}).\]
\end{Pro}

Finally, let us give two useful property of Morse-Novikov homology, the second of which we have hinted at earlier.

\begin{Pro}[\cite{Farber2004TopologyOC}]\label{PropInegM}
	Let $\beta$ be a closed  $1$-form on a connected closed manifold $M$, then \[rank(HN_i(M,\beta))\leq rank(H_i(M,\mathbb{Z})).\]
	
	Moreover, assume that $\beta$, seen as a section in $T^*M$, intersects the $0$ section transversely. Since $\beta=d_\beta(-1)$, we can define the set $Crit_j(\beta)$ of critical points of index $j$ of $\beta$ as the set of $\beta$-critical points of index $j$ of the constant map equal to $-1$. Then we have the lower bound: \[\# Crit_i(\beta)\geq\sum_irank(HN_i(M,\beta))+q_i(M,\beta)+q_{i-1}(M,\beta),\] where $q_i(M,\beta)$ is the minimal number of generators of the torsion subgroup of $rank(HN_i(M,\beta))$, for every $i$.
\end{Pro}

Let us now be a bit more precise with the link between Lichnerowicz cohomology and Morse-Novikov. We have the following lemma, due to A. V. Pazhitnov:
\begin{The}[\cite{Pazhitnov1987APRPNI}]\label{LemPazhitnov1987}
	Assume that $M$ is closed and that the periods of $\beta$ are commensurable. Then, for every $t\in\mathbb{R}$ big enough, and for every $i$,\[dim(H_{t\beta}^i(M))=rank_{Nov(\Gamma_\beta)}(HN^i(M,\beta)).\]
	Note that we are using Morse-Novikov cohomology.
	
	Whenever this equality doesn't hold, the difference between the left hand side and right hand side of this equation is at most the sum of the torsion numbers of the $\mathbb{Q}[\mathcal{H}]$-modules $H_k(\tilde{M}_\beta,\mathbb{Q})$ and $H_{k-1}(\tilde{M}_\beta,\mathbb{Q})$, where $\mathcal{H}$ is the group of deck transformations. Here, $\tilde{M}_\beta$ is the integral cover of $\beta$.
\end{The}
Do note that the above equality is indeed false in some cases, as shown in example 4.2 in A. Moroianu and M. Pilca's 2021 paper (\cite{Moroianu2021}): while related, Lichnerowicz cohomology and Morse-Novikov homology are not the dual of one another. Indeed, in \cite{Moroianu2021}, the authors build a closed nowhere-vanishing $1$-form on a manifold such that its Lichnerowicz cohomology is not $0$. However, since the $1$-form is nowhere-vanishing, the associated Morse-Novikov homology is $0$. This also shows that Lichnerowicz cohomology is less suited at counting the critical points of a closed $1$-form $\beta$ (i.e. the $\beta$-critical points of $d_\beta(-1)=\beta$) than Morse-Novikov (co)homology.

The theorem \ref{LemPazhitnov1987} is a consequence of a more general statement that can also be found in \cite{Pazhitnov1987APRPNI} :

\begin{The}[\cite{Pazhitnov1987APRPNI}] For this theorem, we do not need to assume that $M$ is closed. 
The Lichnerowicz cohomology associated to $-\beta$ is isomorphic to the cohomology with local coefficients in the local system given by the morphism of $\R$-vector space:
\begin{align*}
	\phi\,:\, \mathbb{R}[\pi_1(M)]\rightarrow & \R\\
	\gamma\in\pi_1(M)\mapsto & e^{-<[\beta],h(\gamma)>}
\end{align*}
\end{The}
Indeed, whenever $\Z[<e^g:g\in\Gamma_\beta>]$ is isomorphic as a $\Z[\pi_1(M)]$-module to $\Z[t,t^{-1}]$, the universal coefficient theorems directly implies theorem \ref{LemPazhitnov1987}. 

\begin{Rem}Let $M$ be a manifold, not necessarily closed, and take $\beta\in\Omega^1(M)$ closed such that $[\beta]\in H^1_{dR}(M)\cap H^1(M,\Z)$. Take $\tilde{M}_\beta$ the integral cover of $\beta$ (that is to say the minimal cover on which the pullback of $\beta$ is exact). Then the Morse-Novikov homology is the homology of $C_*(\tilde{M}_\beta)\otimes_{\Z[t,t^{-1}]}\Z[\![t]$ and the Lichnerowicz cohomology is isomorphic to the cohomology of $Hom_{\Z[t,t^{-1}]}(C_*(\tilde{M}_\beta,\R),\R)$.  

Take $S=\{t^r+\sum_{k\in\Z,\,k<r}a_kt^k:r\in\R\textit{ and }\forall k,\,a_k\in\Z\}\cap\Z[t,t^{-1}]$ the set of Laurent polynomials with leading coefficient $1$. As a general property of the Novikov ring, $\Z[\![t]$ is a flat module over the localization $R=S^{-1}\Z[t,t^{-1}]$ and, therefore, Morse-Novikov homology in essence computes $S^{-1}H_*(\tilde{M}_\beta,\Z)\simeq H_*(\tilde{M}_\beta,\Z)\otimes R$ (in as much as the torsion numbers and the ranks are the same).
	
Since $e$ is transcendental, $\Z[<e^g:g\in\Gamma_\beta>]$ is indeed isomorphic to $\Z[t,t^{-1}]$ and its field of fractions $Q$ can be viewed as a subfield of $\R$, and $\R$ is injective as a $Q$-module. Therefore, the Lichnerowicz cohomology, in essence, computes the rank of $Hom_{\Z[t,t^{-1}]}(H_{*}(\tilde{M}_\beta,\Z),\Z[t,t^{-1}])$.

Note that by general properties of the localization for finitely presented modules, \[S^{^-1}Hom_{\Z[t,t^{-1}]}(H_{*}(\tilde{M}_\beta,\Z),\Z[t,t^{-1}])\simeq Hom_{R}(S^{^-1}H_{*}(\tilde{M}_\beta,\Z),R).\] Therefore the torsion-free part of $S^{^-1}H_*(\tilde{M}_\beta,\Z)$ has the same rank (as an $R$-module) as $Hom_{\Z[t,t^{-1}]}(H_{*}(\tilde{M}_\beta,\Z),\Z[t,t^{-1}])$ (as a $\Z[t,t^{-1}]$-module). This implies that the main difference between the Morse-Novikov and Lichnerowicz cohomologies (in our case) comes from \[Ext_1^{R}(S^{-1}H_{*-1}(\tilde{M}_\beta,\Z),R)\simeq S^{-1}Ext_1^{\Z[t,t^{-1}]}(H_{*-1}(\tilde{M}_\beta,\Z),\Z[t,t^{-1}]),\] which gives the torsion submodule of $S^{-1}H_{*-1}(\tilde{M}_\beta,\Z)$.
\begin{flushright}
	\textit{ End of the remark.}
\end{flushright}
\end{Rem}

We can also use some elements of the previous remark to show that Lichnerowicz cohomology is somewhat less suited for providing a lower bound for the number of $\beta$-critical points of a given index even in simple cases.

\begin{Exe}
 	Using the previous remark, take $\Sigma_2$ a genus $2$ surface together with a basis of its first cohomology group. Take $[d\theta]$ some element of the basis and $d\theta$ a $1$-form representing it. If we call $\beta$ the pullback of $d\theta$ to $M=\R P^3\times\Sigma_2$, then the first homology group of $\tilde{M}_\beta$ is isomorphic to $\Z[t,t^{-1}]\oplus(\nicefrac{\Z}{2\Z})[t,t^{-1}]$, and therefore the first Lichnerowicz cohomology group is of rank $1$ (without torsion), while the first Morse-Novikov homology group is of rank $1$, with torsion number given by that of the group $S^{-1}(\nicefrac{\Z}{2\Z}[t,t^{-1}])$, which is generated by one element as an $R$-module.
\end{Exe}

Finally, one may notice that the definition of the Morse-Novikov homology is not too dissimilar to that of the Novikov-Sikorav homology, introduced by J.-C. Sikorav in \cite{Sikorav1987}, who was the first to consider variants of this homology with non-abelian coefficients. Those two homologies can be described as the homology of the chain complex with local coefficients in the ring $\mathcal{R}$, where the morphism \[\rho:\mathbb{Z}[\pi_1(M)]\rightarrow\mathcal{R}\] is such that, for every square matrices $A$ whose entries  $\gamma_{i,j}\in\mathbb{Z}[\pi_1(M)]$ verify $<[\beta],\gamma_{i,j}><0$, $\rho(Id+A)$ is invertible. However, Sikorav's ring is \[\widehat{Z\pi_\beta}=\bigcup_{c\in\mathbb{R}}\left\{\sum n_i\gamma_i\,:\,\gamma_i\in\pi_1(M),n_i\in\mathbb{Z},<[\beta],\gamma_i><c\right\}.\] While this ring is similar to the Novikov ring, notable differences include not being abelian. Both are special cases of  the universal chain complex described by Farber for the first time in \cite{Farber1999MNCPT} (see \cite{Farber2004TopologyOC} for more details). We should point out that both Sikorav and Farber's chain complexes have coefficients in a non-abelian ring, preventing the arguments laid out in the rest of this paper from applying to those homology theories.

\section{A Morse theory for $\beta$-critical points}\label{sec3}

\begin{Not}
	For the rest of this paper, $M$ will be a connected compact manifold.
\end{Not}

As stated in the introduction, we will first prove the following theorem, which is a special case of theorem \ref{thm1} (the case $m=0$).
\begin{The}\label{thmImproveChantraine-Murphy}
	Let $\beta$ be a closed $1$-form on $M$ and $f:M\rightarrow\mathbb{R}$ be a $\beta$-Morse map, then
	\[\# Crit^\beta_{i}(f)\geq rank(HN_i(M,\beta))\]
\end{The}

The proof of this theorem will make heavy use of the following proposition, shown by Farber:
\begin{Pro}(Proposition 1.30, in \cite{Farber2004TopologyOC})
	Let $\beta$ be a closed $1$-form on $M$ and $\tilde{M}_\beta$ be the integral cover of $\beta$. Define \[L_\beta:=\mathbb{R}[H_1(M)/ker(<[\beta],\cdot>)]\] and let $Q_\beta$ be the field of fractions of $L_\beta$.
	
	Then, for every $i$, we have: \[rank(HN_i(M,\beta))=rank(Q_\beta\otimes_{L_\beta}H_i(\tilde{M}_\beta,\mathbb{R})).\]
\end{Pro}

The proof of this theorem will be broken down in three main acts, with an intermezzo in subsection \ref{MorseCob} for a brief discussion about Morse theory for cobordisms ``with boundaries'' (aka. cobordisms between manifolds with boundaries). In the first subsection, we will see $L_\beta$ acts on the cover, ultimately leading to corollary \ref{CorMorseNovIneg}, which will close this subsection by establishing a lower bound on the asymptotic behavior of the homology of progressively increasing compact sets of the cover. After the intermezzo, subsection \ref{relevéfonction} will lead with a construction modifying the map $f$ of the theorem above to allow it to be studied through Morse theory (for cobordisms with boundaries). It will conclude by giving an asymptotic upper bound on how many critical points the construction adds to $f$. Our final act will be to put subsection \ref{rank} and \ref{relevéfonction} together to obtain the result in subsection \ref{lower}.

\subsection{A couple of inequalities of ranks}\label{rank}

 Let us use the beginning of Farber's construction for his universal chain complex (see \cite{Farber2004TopologyOC}). More specifically, we are interested in his construction of a ``good'' fundamental domain in a covering space of $M$.

\begin{Rem}[\cite{Farber2004TopologyOC}]
Let $\beta$ be a closed $1$-form on $M$, then there is an $r\in\mathbb{N}$ and a map \[(\psi_1,\ldots,\psi_r):M\rightarrow\mathbb{S}^1_1\times\ldots\times\mathbb{S}^1_r\] such that $[\beta]=\sum_ia_i\psi_i^{*}[d\theta_i]$ where $d\theta_i$ is the canonical generator of $H^1(\mathbb{S}^1_i)$, the $\psi_i^{*}[d\theta_i]$ are linearly independent and the $a_i$ are linearly independent over $\mathbb{Z}$. 
\end{Rem}
Moreover, due to Thom's transversality, we have:
\begin{Rem}[\cite{Farber2004TopologyOC}]
For some generic $\beta$, generic elements $c_1,\ldots,c_r\in\mathbb{S}^1$, and for every $i$, the various maps $\phi_i^{-1}(\{c_i\})$ are submanifold which intersect transversely two by two.
\end{Rem}
These consideration allow for the statement of the following definition:
\begin{Def}[\cite{Farber2004TopologyOC}]
Let $V=M-\bigcup_i\phi_i^{-1}(\{c_i\})$, take $\tilde{V}_\beta$ a preimage of $V$ in $\tilde{M}_\beta$ which is fully contained in some connected fundamental domain (with corners).
\end{Def}
\begin{figure}[H]
\centering
\includegraphics{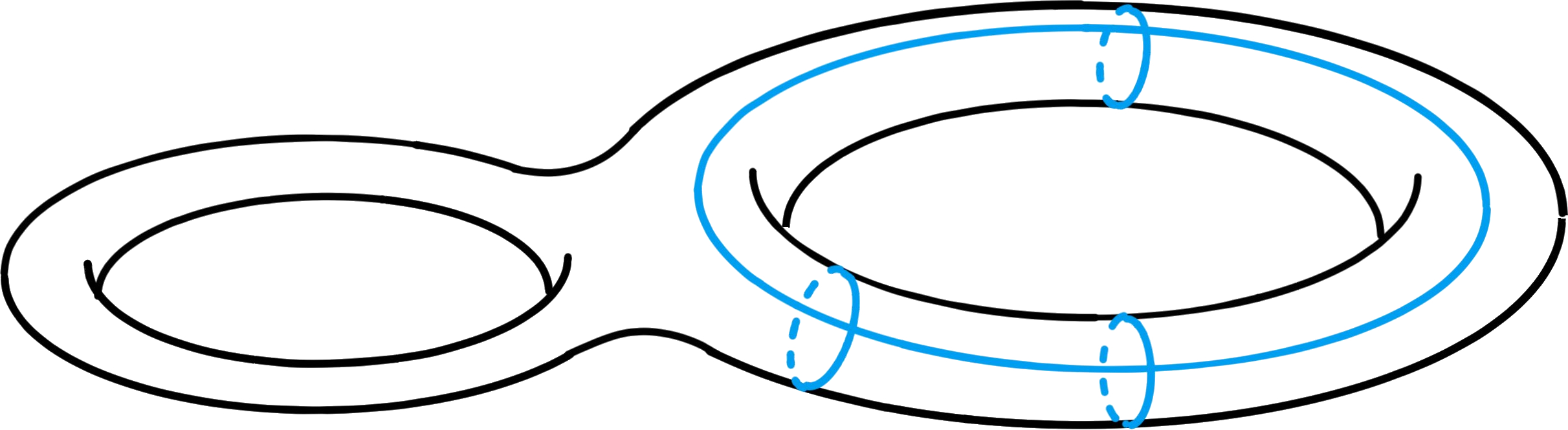}
\caption{Example of $V$ for a genus $2$ surface.}
\end{figure}

\begin{Rem}
Remark that $\overline{\tilde{V}_\beta}$, the closure of $\tilde{V}_\beta$, is a fundamental domain (give or take some boundaries) with corner. Moreover, if $g$ is a primitive of $\beta$ in $\tilde{M}_\beta$, then $\overline{\tilde{V}_\beta}$ has a positive boundary (with $dg$ exiting) $\partial_+\overline{\tilde{V}_\beta}$ and negative boundary ($dg$ entering) $\partial_-\overline{\tilde{V}_\beta}$, the positive and negative boundaries intersect along corners. Finally, observe that by construction the projection of $\partial\overline{\tilde{V}_\beta}$ in $M$ is a subset of $\bigcup_i\phi_i^{-1}(\{c_i\})$.
\end{Rem}

Endow $\tilde{M}_\beta$ with a CW structure such that $\mathcal{H}$, the automorphism group of $\tilde{M}_\beta$, acts on the structure and $\overline{\tilde{V}_\beta}$ is a sub-complex. Since $C_*^{CW}(\tilde{M}_\beta,\mathbb{R})$ is free $L_\beta$-module, we have the following equalities, partially give by the universal coefficient theorem (fore example found in theorem 3.6.1 in \cite{Weibel1994AnIHA}):
$rank(Q_\beta\otimes_{L_\beta} H_i(\tilde{M}_\beta,\mathbb{R})= rank (H_i(Q_\beta\otimes_{L_\beta} C_*^{CW}(\tilde{M}_\beta,\mathbb{R}))) $ $= rank(H_i(Q_\beta\otimes_{\mathbb{R}}C_*^{CW}(\overline{\tilde{V}_\beta},\partial_-\overline{\tilde{V}_\beta}),\hat{\delta}))$. The differential $\hat{\delta}$ 
isn't the usual one. Indeed, it behaves like the usual one for the cells in the interior of  $\overline{\tilde{V}_\beta}$, but identifies the cells $c$ in $\partial_-\overline{\tilde{V}_\beta}$ to the cells $\alpha^{-1}\otimes_\mathbb{R}\alpha(c)\in Q_\beta\otimes_{\mathbb{R}}C_*^{CW}(\overline{\tilde{V}_\beta},\partial_-\overline{\tilde{V}_\beta})$, with $\alpha(c)$ in $\partial_-\overline{\tilde{V}_\beta}$ through the action of $\mathcal{H}$. More precisely, the differential $\hat{\delta}$ can be written as $\partial+\partial'$ where $\partial$ is the usual differential, and $\partial'$ is $0$ on the cells which are not in $\partial_-\overline{\tilde{V}_\beta}$, and equal to $\alpha^{-1}\otimes\alpha(c)$ otherwise, with $\alpha(c)\subset\partial_-\overline{\tilde{V}_\beta}$. Moreover, the degree of the polynomials in $L_\beta$ induce a filtration on $L_\beta\otimes_{\mathbb{R}}C_*^{CW}(\overline{\tilde{V}_\beta},\partial_-\overline{\tilde{V}_\beta})$. This filtration, in turn, induce a spectral sequence the $E_2$ page of which is 
$L_\beta\otimes H_i(\overline{\tilde{V}_\beta},\partial_-\overline{\tilde{V}_\beta})$ and with $E_3$ page equal to $H_i(L_\beta\otimes_{\mathbb{R}}C_*^{CW}(\overline{\tilde{V}_\beta},\partial_-\overline{\tilde{V}_\beta}),\hat{\delta})$.
In particular,
\begin{align*}
	rank(H_i(\overline{\tilde{V}_\beta},\partial_-\overline{\tilde{V}_\beta}))&=rank(L_\beta\otimes H_i(\overline{\tilde{V}_\beta},\partial_-\overline{\tilde{V}_\beta}))\\
	&\geq rank(H_i(L_\beta\otimes_{\mathbb{R}}C_*^{CW}(\overline{\tilde{V}_\beta},\partial_-\overline{\tilde{V}_\beta}),\hat{\delta})).
\end{align*} 

We will now seek to generalize this observation. Let us start with a couple of definitions.

\begin{Def}
Let $(\alpha_0,\ldots,\alpha_r)$ be a basis of \[\mathcal{H}\simeq H_1(M)/ker(<[\beta],\cdot>)\] such that $<\beta,\alpha_i>\geq0$ for every $i$:
\begin{itemize}
\item Take $\mathcal{H}_k=\{\sum_{i=0}^r\lambda_i\alpha_i\,:\,\forall i,\lambda_i\in\{-k,\ldots,k\}\}$.
\item Take $W_k=\cup_{\alpha\in\mathcal{H}_k}\alpha(\overline{\tilde{V}_\beta})$. 
\item Take $\partial_+\mathcal{H}_k=\{\{\sum_{i=0}^r\lambda_i\alpha_i\in\mathcal{H}_k\,:\,\exists i/\lambda_i=k\}\}$ and $\partial_-\mathcal{H}_k=\{\sum_{i=0}^r\lambda_i\alpha_i\in\mathcal{H}_k\,:\,\exists i/\lambda_i=-k\}$. Note that those two sets are not necessarily disjoint. 
\end{itemize}
\end{Def}
\begin{Rem}
We not here that $W_k$ has a negative boundary ($dg$ enters) $\partial_-W_k$ and a positive one ($dg$ exits) $\partial_+W_k$ which intersect in corners. Then $\partial_+W_k\subset\cup_{\alpha\in\partial_+\mathcal{H}_k}\alpha(\partial_-\overline{\tilde{V}_\beta})$ and $\partial_-W_k\subset\cup_{\alpha\in\partial_-\mathcal{H}_k}\alpha(\partial_-\overline{\tilde{V}_\beta})$.
\end{Rem}
However, for simplicity's sake, we will restrict ourselves to ``negative'' automorphisms.
\begin{Def}
	Let $(\alpha_0,\ldots,\alpha_r)$ be the same basis as in the previous definition.
\begin{itemize}
\item Take $\mathcal{H}^-=\{\sum_{i=0}^r\lambda_i\alpha_i\in\mathcal{H}\,:\, \lambda_i\leq 0\}$ and $\mathcal{H}_k^-=\mathcal{H}_k\cap\mathcal{H}^-$.
\item Take $W_k^-=\cup_{\alpha\in\mathcal{H}_k^-}\alpha(\overline{\tilde{V}_\beta})$.
\item Take $\partial_-\mathcal{H}^-_k=\{\sum_{i=0}^r\lambda_i\alpha_i\in\mathcal{H}^-_k\,:\,\exists i/\lambda_i=-k\}$.
\item Take $\partial_+\mathcal{H}^-_k=\{\sum_{i=0}^r\lambda_i\alpha_i\in\mathcal{H}^-_k\,:\,\exists i/\lambda_i=0\}$.
\item $L_\beta^-:=\mathbb{R}[\mathcal{H}^-]$ and $L_\beta^k:=\mathbb{R}[\mathcal{H}^-]/I$ where $I$ is the ideal in $\mathbb{R}[\mathcal{H}^-]$ generated by the elements of $\partial_-\mathcal{H}_{k+1}^-$.
\end{itemize}
\end{Def}

\begin{Rem}\label{Remegal}
Similarly to $W_k$,we define $\partial_+W_k^-$ and $\partial_-W_k^-$. Therefore, $\partial_-W_k^-\subset\cup_{\alpha\in\partial_-\mathcal{H}^-_k}\alpha(\overline{\partial_-\tilde{V}_\beta})$ and $\partial_+W_k^-\subset\cup_{\alpha\in\partial_+\mathcal{H}^-_k}\alpha(\overline{\partial_+\tilde{V}_\beta})$. 

Thus, we can point out that $C^{CW}_*(W^-_k,\partial_-W^-_k)$ is a free $L_\beta^k$-module equal to $L_\beta^k\otimes_\mathbb{R}C_*^{CW}(\overline{\tilde{V}_\beta},\partial_-\overline{\tilde{V}_\beta})$ endowed with the differential $\hat{\delta}$. 
\end{Rem}

\begin{figure}[H]
	\centering
	\begin{overpic}[scale=0.65]{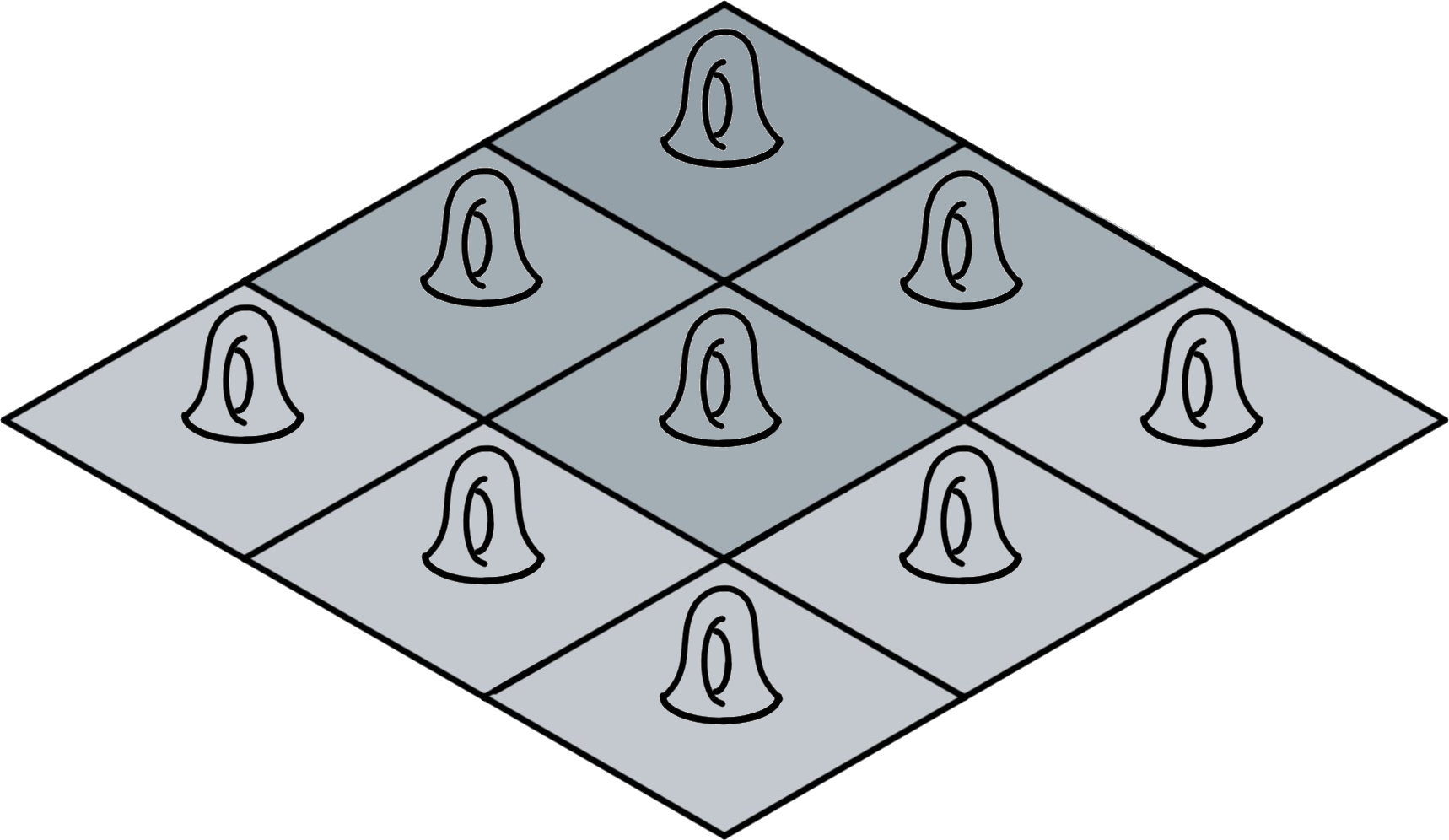}
		\put (60,52) {$W_0^-$}
		\put (76,43) {$W_1^-$}
		\put (93,33) {$W_2^-$}
	\end{overpic}
	\caption{Example of $W_0^-$, $W_1^-$ and $W_2^-$.}
\end{figure}

We can now state the following generalization of our previous observation:
\begin{Pro}
	$(k+1)^{r+1}rank(HN_i(M,\beta))\leq rank(H_i(W^-_k,\partial_-W^-_k))$, with $r+1=rank_\mathbb{Z}(\mathcal{H})$.	
\end{Pro}
The reminder of this subsection will be essentially taken by the proof of this proposition.
\begin{proof}
First, let us show the following lemma:
\begin{Lem}\label{StupidLemma} Let $([c_i])_{i=1,\ldots p}$ be a free family $L_\beta^k$-module \[H_i(L_\beta^k\otimes_{L_\beta^k}C_*^{CW}(W_k^-,\partial_-W_k^-)),\] then
\[	dim_\mathbb{R}(L_\beta^k)\times p\leq dim_\mathbb{R} (H_i(W_k^-,\partial_-W^-_k))\]
\end{Lem}
\begin{proof}
	Let $\psi: L_\beta^k\otimes_{L_\beta^k}C_*^{CW}(W_k^-,\partial_-W^-_k)\rightarrow C_*^{CW}(W_k^-,\partial_-W^-_k)$ be the canonical isomorphism. We can see that $\psi$ commutes with the differential, an thus $\psi$ is a chain complex isomorphism, and induces an isomorphism in homology. More specifically, since $\psi$ is a morphism of $\mathbb{R}$-modules, then for every free family $([c_i])_i$ of elements  $[c_i]\in H_i(L_\beta^k\otimes_{L_\beta^k}C_*^{CW}(W_k^-,\partial_-W^-_k))$, $(\psi([c_i]))_i$ is a free family. Take $P_i$ a family of non-zero elements of $L_\beta^k$, then $\psi([P_i\times c_i])=[P_i\psi(c_i)]$ given by the action of $P_i$ on $\psi(c_i)$, and $[P_i c_i]$ is also forms free family. Therefore, $(\psi([P_i\times c_i]))_i$ is also a free family.
\end{proof}
In particular, this implies the following equality:
\begin{equation}\label{inégalités2.1.1.1}
dim_\mathbb{R}(L_\beta^k)rank_{L_\beta^k}\big(H_i(L_\beta^k\otimes_{L_\beta^k}C_*^{CW}(W_k^-,\partial_-W^-_k))\big)\leq dim_\mathbb{R} (H_i(W_k^-,\partial_-W^-_k))
\end{equation} 

On the other hand, the universal coefficient theorem implies that we have an injection:
\[ L_\beta^k\otimes_{L_\beta^-}H_i(L_\beta^-\otimes_\mathbb{R}C_*^{CW}(\overline{\tilde{V}_\beta},\partial_-\overline{\tilde{V}_\beta}),\hat{\delta})\hookrightarrow H_i(L_\beta^k\otimes_\mathbb{R}C_*^{CW}(\overline{\tilde{V}_\beta},\partial_-\overline{\tilde{V}_\beta}),\hat{\delta}),\]
This means that we have the following inequalities of ranks:
\begin{align}\label{inégalités2.1.1.2}
	rank_{L_\beta^-}\big(H_i(L_\beta^-\otimes_\mathbb{R}C_*^{CW}(&\overline{\tilde{V}_\beta},\partial_-\overline{\tilde{V}_\beta}),\hat{\delta}))\big)\\
	&\leq rank_{L_\beta^k}(
	L_\beta^k\otimes_{L_\beta^-}H_i(L_\beta^-\otimes_\mathbb{R}C_*^{CW}(\overline{\tilde{V}_\beta},\partial_-\overline{\tilde{V}_\beta}),\hat{\delta}))\nonumber\\
	&\leq
	rank_{L_\beta^k}( H_i(L_\beta^k\otimes_\mathbb{R}C_*^{CW}(\overline{\tilde{V}_\beta},\partial_-\overline{\tilde{V}_\beta}),\hat{\delta}))\nonumber\\
	&\leq rank_{L_\beta^k}\big(H_i(L_\beta^k\otimes_{L_\beta^k}C_*^{CW}(W_k^-,\partial_-W^-_k))\big)
\end{align}
where the last inequality stems from the fact that $\left(L_\beta^k\otimes_\mathbb{R}C_*^{CW}(\overline{\tilde{V}_\beta},\partial_-\overline{\tilde{V}_\beta}),\hat{\delta}\right)$ and $C_*^{CW}(W_k^-,\partial_-W^-_k)$ are equal as chain complexes of $L_\beta^k$-modules (see remark \ref{Remegal})

Therefore, the inequalities (\ref{inégalités2.1.1.1}) and (\ref{inégalités2.1.1.2}) imply that:
\begin{equation}\label{inégalités2.1.1.3}
dim_\mathbb{R}(L_\beta^k)rank_{L_\beta^-}\big(H_i(L_\beta^-\otimes_\mathbb{R}C_*^{CW}(\overline{\tilde{V}_\beta},\partial_-\overline{\tilde{V}_\beta}),\hat{\delta}))\big)\leq dim_\mathbb{R} (H_i(W_k^-,\partial_-W^-_k))
\end{equation}

The link between the previous inequality and the Morse-Novikov homology is given by the following lemma:

\begin{Lem}
	$rank_{L_\beta}(H_i(L_\beta\otimes_\mathbb{R}C_*^{CW}(\overline{\tilde{V}_\beta},\partial_-\overline{\tilde{V}_\beta}),\hat{\delta}))= rank_{L_\beta^-}(H_i(L_\beta^-\otimes_\mathbb{R}C_*^{CW}(\overline{\tilde{V}_\beta},\partial_-\overline{\tilde{V}_\beta}),\hat{\delta}))$
\end{Lem}
\begin{proof}
	Notice that $L_\beta$ and $L_\beta^-$ have the same field of fractions $Q_\beta$. therefore, by the universal coefficient theorem, we have the following isomorphisms:
	\[ Q_\beta\otimes_{L_\beta^-}H_i(L_\beta^-\otimes_\mathbb{R}C_*^{CW}(\overline{\tilde{V}_\beta},\partial_-\overline{\tilde{V}_\beta}),\hat{\delta})\rightarrow H_i(Q_\beta\otimes_\mathbb{R}C_*^{CW}(\overline{\tilde{V}_\beta},\partial_-\overline{\tilde{V}_\beta}),\hat{\delta});\]
	and
	\[ Q_\beta\otimes_{L_\beta}H_i(L_\beta\otimes_\mathbb{R}C_*^{CW}(\overline{\tilde{V}_\beta},\partial_-\overline{\tilde{V}_\beta}),\hat{\delta})\rightarrow H_i(Q_\beta\otimes_\mathbb{R}C_*^{CW}(\overline{\tilde{V}_\beta},\partial_-\overline{\tilde{V}_\beta}),\hat{\delta});\]
	Therefore,
	\begin{align*}
	dim(H_i(Q_\beta\otimes_\mathbb{R}C_*^{CW}(\overline{\tilde{V}_\beta},\partial_-\overline{\tilde{V}_\beta}),\hat{\delta}))&=rank_{L_\beta^-}(H_i(L_\beta^-\otimes_\mathbb{R}C_*^{CW}(\overline{\tilde{V}_\beta},\partial_-\overline{\tilde{V}_\beta}),\hat{\delta}))\\
	&=rank_{L_\beta}(H_i(L_\beta\otimes_\mathbb{R}C_*^{CW}(\overline{\tilde{V}_\beta},\partial_-\overline{\tilde{V}_\beta}),\hat{\delta})).
	\end{align*}
\end{proof}
Now we use the universal coefficient theorem to derive:
\begin{align*}
	dim_{Nov(\Gamma_\beta)}(HN_i(M,\beta))&= 
	dim_{Q_\beta}(H_i(Q_\beta\otimes_\mathbb{R}C_*^{CW}(\overline{\tilde{V}_\beta},\partial_-\overline{\tilde{V}_\beta}),\hat{\delta}))\\&=rank_{L_\beta}(H_i(L_\beta\otimes_\mathbb{R}C_*^{CW}(\overline{\tilde{V}_\beta},\partial_-\overline{\tilde{V}_\beta}),\hat{\delta}))
\end{align*}
Therefore, putting the previous lemma and inequality (\ref{inégalités2.1.1.3}) together imply that:
\[dim_\mathbb{R}(L_\beta^k)dim_{Nov(\Gamma_\beta)}(HN_i(M,\beta))\leq dim_\mathbb{R} (H_i(W_k^-,\partial_-W^-_k))\]

Finally, notice that $dim_\mathbb{R}(L_\beta^k)=(k+1)^{r+1}$ (volume of a hypercube of dimension $r+1$ and with side of length $k+1$), thus showing the proposition.\end{proof}

Observe that $W^-_{2k+1}=\alpha_0^{-k}\circ\ldots\circ\alpha_r^{-k}( W_k)$.This leads us to the following corollary:

\begin{Cor}\label{CorMorseNovIneg}
	$(2k+1)^{r+1}rank(HN_i(M,\beta))\leq rank(H_i(W_k,\partial_-W_k))$
\end{Cor}
\begin{proof}
	By the two previous lemmas, \[rank(HN_i(M,\beta))\leq rank_{L_\beta^{2k+1}}(H_i(L_\beta^{2k+1}\otimes_{L_\beta^{2k+1}}C_*^{CW}(W_{2k+1}^-,\partial_-W_{2k+1}^-)))\] and \[(2k+1)^{r+1} rank_{L_\beta^{2k+1}}(H_i(L_\beta^{2k+1}\otimes_{L_\beta^{2k+1}}C_*^{CW}(W_{2k+1}^-,\partial_-W_{2k+1}^-)))\]\[\leq rank_\mathbb{R}(H_i(W_{2k+1}^-,\partial_-W_{2k+1}^-)).\]
	This allows us to conclude by pointing out that: \[H_i(W_{2k+1}^-,\partial_-W_{2k+1}^-)=H_i(W_{k},\partial_-W_{k}).\]
\end{proof}

\subsection{\centering Morse theory for cobordisms between manifolds with boundary}\label{MorseCob}

Let $N_1,N_2$ be two compact manifolds with diffeomorphic boundary that we will denote by $\partial N$, $C$ be a cobordism from $N_1$ to $N_2$ such that $\partial C= N_1\times\{0\}\cup N_2\times\{1\}\cup\partial N\times[0,1]$. We will consider Morse functions $f$ such that 
\begin{enumerate}
\item $\forall x\in C-N_1\times\{0\}, f(x)>inf(f)$,
\item $\forall x\in C-N_2\times\{1\}, f(x)<sup(f)$,
\item  $f_{|N_1\times\{0\}}=inf(f)$, $f_{|N_2\times\{1\}}=sup(f)$.
\end{enumerate}
Assume moreover that the differential $df$ is non-zero on the boundary, and that the gradient of $f$ is tangent to $\partial N\times [0,1]$ for some Riemannian metric. If we want to drop the compacity condition on $C$, we need to ask that $f$ only has a finite number of critical points.

We can glue $C$ to itself along $\partial N\times[0,1]$ to make a new cobordism $C'$ from $N_1\cup_{\partial N}N_1$ to $N_2\cup_{\partial N}N_2$. We can also take $f'$ the obvious extension of $f$ to $C'$. Then, for every $i$, \[\# Crit_j(f')=2\# Crit_j(f)\geq rank(H_j(C',N_1\cup_{\partial N}N_1\times\{0\})).\]

Moreover, we have the Mayer-Vitoris sequence:
\begin{align*}
	\ldots\overset{-1}{\rightarrow}H_j(\partial N\times[0,1],\partial N\times\{0\})\rightarrow& H_j(C,N_1\times\{0\})+H_j(C,N_1\times\{0\})\\
	& \rightarrow H_j(C',(N_1\cup_{\partial N}N_1)\times\{0\})\overset{-1}{\rightarrow}\ldots
\end{align*}
Therefore, $2rank( H_j(C,N_1\times\{0\}))=rank(H_j(C',N_1\cup_{\partial N}N_1\times\{0\}))$, yielding: \[\# Crit_j(f)\geq rank( H_j(C,N_1\times\{0\})).\]

\centering\subsection{On the lift of $\beta$-Morse functions}\label{relevéfonction}

\raggedright
\begin{Not}
For the rest of this paper, the manifold $M$ will also be closed
\end{Not}

Let $f$ be a function on some closed manifold $M$ and let $\beta$ be a closed $1$-form on $M$. Assume that $f$ is $\beta$-Morse. Take $g$ a primitive of $\beta$ on $\tilde{M}_\beta$ the integral cover of $\beta$.
Let $f$ be the lift of $f$ to $\tilde{M}_\beta$. Assume that the fundamental domain $\tilde{V}_\beta$ has been chosen such that $e^{-g}f$ has no critical point on $\partial\overline{\tilde{V}_\beta}$. Notice that in essence, $W_k$ is a cobordism $C$ from $N_1:=\partial_- W_k$ to $N_2:=\partial_+W_k\simeq\partial_-W_k$, with the corners of $W_k$ acting as the ``boundary'' of the cobordism (that is, $\partial C-(N_1\cup N_2)$), which is isomorphic to $N_1\times[0,1]$. 

Therefore, we will extend $(e^{-g}f)_{|W_k}$ in four steps so that it fulfills the conditions outlined in the previous subsection. We will focus on constructing the extension on $\partial_-W_k$, but the same process can be used to extend along the other boundary.\newline

\paragraph{\textbf{Extension 1.}}
Let $x_0\in\partial_-\overline{\tilde{V}_\beta}$ such that $f(x_0)=\inf\left(f_{|\partial_-\overline{\tilde{V}_\beta}}\right)$ and $\alpha\in\mathcal{H}$ such that $\alpha\neq id$. On $\alpha(\partial_-\overline{\tilde{V}_\beta})$, define: \[h^-_\alpha(y)=e^{-\int_\gamma dg}e^{-\int_\alpha\beta}e^{-g(x_0)}f(x_0)=e^{-g(y)+g(x_0)}e^{-g(x_0)}f(x_0)=e^{-g(y)}f(x_0)\] with $\gamma$ some path from $\alpha\cdot x_0$ to $y$.
Let $\phi:\mathbb{R}\rightarrow[0,1]$ be a map with support in $[0,1]$ of positive derivative and such that $\phi(0)=0$, $\phi(1)=1$ and $\phi'_{|t=0}=\phi'_{|t=1}=0$. Take: \[H^-_\alpha(y,t)=\phi(t)(f(y)-(1-t)df_{|y}(\partial_t))+(1-\phi(t))(h(y)+t).\] Notice that $H_\alpha$ is $C^1$ and that the places where the differential of $H^-_\alpha$ is $0$ do not depend on $\alpha$.

Do take note that the differential of $H^-_\alpha$ enters along $\partial_-\overline{\tilde{V}_\beta}\times\{0\}$. Similarly, we can define $H^+_\alpha$ by swapping $\inf$ with $\sup$ to extend $f$ along $\alpha(\partial_+\overline{\tilde{V}_\beta})$ such that the differential of $H^+_\alpha$ exits along $\partial_+\overline{\tilde{V}_\beta}\times\{1\}$. We can now glue together the various maps $H^\pm_\alpha$ along the boundary of $W_k$ to form our first extension, $H$, defined on \[(\partial_-W_k\times[0,1])\cup_{\partial_-W_k\times\{1\}}W_k\cup_{\partial_+W_k\times\{0\}}(\partial_+W_k\times[0,1]).\] 
This map $H$ is of class $C^1$ on its domain of definition, and of class $C^2$ on \[(\partial_-W_k\times[0,1[)\sqcup int(W_k)\sqcup (\partial_+W_k\times]0,1]),\]
where $int$ is the interior. As all the critical points of $H$ are placed where $H$ is $C^2$, we may assume that $H$ is a smooth Morse function. 

The total number of critical points added at this stage is therefore a multiple of this size of how many copies of $\tilde{V}_\alpha$ are in the boundary of $W_k$.\newline

\paragraph{\textbf{Extension 2.}}
Now take the map: \[g_\alpha^-(y)=e^{-\int_\gamma dg}e^{-\int_\alpha \beta}\inf\left(e^{-g}_{|\partial_-\overline{\tilde{V}_\beta}}\right)f(x_0).\] 
We then do a similar extension as above between $H$ and the various $g^-_\alpha$ to give an extension $G^-$. Note that this extension does not add critical points since the differential of $H^-_\alpha$ enters along $\partial_-\overline{\tilde{V}_\beta}\times\{0\}$. Similarly, we can define $G^+_\alpha$ by swapping $\inf$ with $\sup$. We do the same kind of extension as above by glueing the various $G^\pm_\alpha$. This defines an extension $G$ on \[(\partial_-W_k\times[0,1])\cup_{\partial_-W_k\times\{1\}}W_k\cup_{\partial_+W_k\times\{0\}}(\partial_+W_k\times[0,1])\] that we may assume to be smooth.\newline

\paragraph{\textbf{Extension 3.}}
Now take the map: \[m_k(y)=\inf_{\alpha\in\mathcal{H}_k,y\in\alpha(\partial_-\overline{\tilde{V}_\beta})}(g^-_\alpha(y))\times\inf\left(e^{-g}_{|\partial_-\overline{\tilde{V}_\beta}}\right)f(x_0)-1,\] 
We then do a similar extension as above between $G^-$ and $m$ to give an extension $l^-$. Note that this extension does not add critical points. Similarly, we can define $l^+_\alpha$ by swapping $\inf$ and $-1$ with $\sup$ and $+1$. Call $L$ the total extension of $e^{-g}f$ done at that stage, which is defined on $(\partial_-W_k\times[0,1])\cup_{\partial_-W_k\times\{1\}}W_k\cup_{\partial_-W_k\times\{0\}}(\partial_+W_k\times[0,1])$. \newline

\paragraph{\textbf{Extension 4.}}
Notice that at this stage, $L$ almost fulfills the conditions of the previous subsection. However, there might be some problems along the corners of $W_k$, which correspond to the boundary of the cobordism.

On $(\partial_-W_k\cap\partial_+ W_k)\times[0,1]=:\partial\partial W_k\times[0,1]$, take the map: \[c(t)=(1-t)(\inf(L)-1+t)+t(\sup(L)+t),\]
and then take an interpolation (at least $C^1$) between $c$ and $L$ along $(\partial_-W_k\cap\partial_- W_k)\times[0,1]$ as previously done. This last interpolation might add a number of critical points at each corners (independent of which corner is considered), but no more than some multiple of the number of
the number of copies of $\tilde{V}_\alpha$ that are in the boundary of $W_k$.

Call $F_k$ the total extension of $(e^{-g}f)_{|W_k}$, and call the $\tilde{F}_k$ the part of the map that is not equal to $e^{-g}f$. The map $F_k$ is defined on 
\begin{align*}
	C_k:=\big[(\partial_-W_k\times[0,1])&\cup_{\partial_-W_k\times\{1\}}W_k\cup_{\partial_+W_k\times\{0\}}(\partial_+W_k\times[0,1])\big]\\
	&\cup_{\{1\}\times(\partial\partial W_k)\times[0,1]}\big[[0,1]\times(\partial\partial W_k)\times[0,1]\big],
\end{align*}
while $\tilde{F}_k$ is defined on  $C_k-W_k$. Note that $\partial C_k=((\partial N)\times [0,1])\cup (N\times\partial[0,1])$ for $N\simeq\partial_- W_k\cup_{\{1\}\times(\partial\partial_- W_k)}[0,1]\times (\partial\partial_-W_k)$.\newline

\paragraph{\textbf{Final observation.}} Note that $\mathcal{H}_k$ is a cube of dimension $2k+1$ and of side length $r+1$. As such there are two constants $K^j_1$ and $K^j_2$ depending only on $j$ such that (keeping in mind the previous subsection):
\begin{align*}\# Crit_j(F_k) &=\# Crit_j((e^{-g}f)_{|W_k})+\# Crit_j(\tilde{F}_k)\\
&\leq\# Crit_j((e^{-g}f)_{|W_k})+ 2(r+1)(2k+1)^rK^j_1\\&\qquad\qquad\qquad\qquad+2r(r+1)(2k+1)^{r-1} K^j_2\\
&=\# Crit_j((e^{-g}f)_{|W_k})+ O_{k\rightarrow +\infty}((2k+1)^r)\end{align*}

\subsection{A lower bound for $\# Crit_\beta(f)$} \label{lower}

Putting together th previous subsection and the corollary \ref{CorMorseNovIneg}, we have the following (in)equalities:
\begin{align*}
\# Crit_j(F_k) & =	\# Crit_j((e^{-g}f)_{|W_k})+ O_{k\rightarrow +\infty}((2k+1)^r)\\
	& =  (2k+1)^{r+1}\# Crit_j((e^{-g}f)_{|W_0})+ O_{k\rightarrow +\infty}((2k+1)^r)\\
	&  \geq rank(H_j(W_k,\partial_-W_k))\\
	& \geq (2k+1)^{r+1}rank(HN_j(M,\beta))
\end{align*}

Therefore,
\[\# Crit_j((e^{-g}f)_{|W_0})+ O_{k\rightarrow +\infty}((2k+1)^{-1})\geq rank(HN_j(M,\beta)),\]

and, for $k\rightarrow+\infty$, we get, \[\# Crit_j((e^{-g}f)_{|W_0})\geq rank(HN_j(M,\beta)).\]

Since $e^{-g}d_{dg}f=de^{-g}f$, the critical points of $(e^{-g}f)_{|W_0}$ are exactly the $\beta$-critical points of $f$, yielding the theorem \ref{thmImproveChantraine-Murphy}.

\section{A generalization to generating functions}\label{sec4}

Before proving theorem \ref{thm1}, let us give a refresher on what is a generating function.

\begin{Def}
Take a closed manifold $M$ and a smooth map $F:M\times\mathbb{R}^m\rightarrow\R$ for some $m\geq0$. We say that $F$ is a generating function if there is some compact $K\subset\R^m$ and some quadratic map $Q:\R^m\rightarrow\R$ such that $F=Q$ outside of $M\times K$.
\end{Def}

\begin{proof}[proof of theorem \ref{thm1}]
With the same notations as in the definition above.

 Take $a$ big enough and $b$ small enough so that $F$ is quadratic on: \[M\times\{\xi\in\mathbb{R}^m\;:\;Q(\xi)<b\textit{ or }a<Q(\xi)\}.\]
Write $E_b^a=\{\xi\in\mathbb{R}^m\;:\;b<Q(\xi)<a\}$. Abusing the notations slightly, the pullback of a quantity to $\tilde{M}_\beta\times\mathbb{R}^m$ will be named the same. Notice that by defining:
\begin{align*}
\text{1. }\partial_+ (W_k\times E_b^a)&=(\partial_+W_k)\times E_b^a\cup W_k\times\{\xi\in\mathbb{R}^m\;:\;Q(\xi)=a\},\\
\text{2. }\partial_- (W_k\times E_b^a)&=(\partial_-W_k)\times E_b^a\cup W_k\times\{\xi\in\mathbb{R}^m\;:\;Q(\xi)=b\},\\
\text{3. }\partial_+ (W_k\times E_b^a)&\cap\partial_-(W_k\times E_b^a)=(\partial_+W_k\cap\partial_-W_k)\times E_b^a,
\end{align*}
we can apply the construction of the subsection \ref{relevéfonction} to $e^{-g}F$, where $g$ is a primitive of $\beta$ on $\tilde{M}_\beta\times\mathbb{R}^m$.

Since $de^{-g}Q$ enters along $W_k\times\{\xi\in\mathbb{R}^m\;:\;Q(\xi)=b\}$ and exits along $W_k\times\{\xi\in\mathbb{R}^m\;:\;Q(\xi)=a\}$, then, just as previously, the construction adds at most $O_{k\rightarrow+\infty}(k^r)$ critical points where $r+1=rank_\mathbb{Z}(\mathcal{H})$.

Therefore, by applying subsection \ref{rank} to $M\times E_b^a$ and classical Morse theory, we have that: \[Crit_j(e^{-g}F_{|W_k\times\mathbb{R}^m})+O_{+\infty}(k^r)\geq rank\left(H_j\left(W_k\times E_b^a,\partial_-(W_k\times E_b^a)\right)\right)
.\]
Call $H$ the vector subspace on which $Q$ is negative definite, and $H_b^a=H\cap E_b^a$. Then, by the excision theorem, we have that:
\begin{align*}
H_j\big(W_k\times E_b^a,\partial_-(W_k&\times E_b^a)\big)= H_j\left(W_k\times E_b^a,(\partial_-W_k)\times E_b^a\cup W_k\times Q^{-1}(\{b\})\right) \\
&\simeq H_j\left(W_k\times H^a_{b},(\partial_-W_k)\times H^a_{-\infty}\cup W_k\times (H\cap Q^{-1}(\{b\}))\right)\\
&\simeq H_j\left(W_k\times H,(\partial_-W_k)\times H\cup W_k\times H^b_{-\infty}\right)
\end{align*}
Write:
\begin{align*}
	\text{1. }A_j&=H_j\left((\partial_-W_k)\times H, (\partial_-W_k)\times H^b_{-\infty}\right),\\
	\text{2. }B_j&=H_j\left(W_k\times H, W_k\times H^b_{-\infty}\right),\\
	\text{3. }C_j&=H_j\left(W_k\times H,(\partial_-W_k)\times H\cup W_k\times H^b_{-\infty}\right).
\end{align*}

Then, taking $dim(H)=p$, we have the following diagram:
\begin{figure}[H]
\centering
\begin{tikzpicture}
	\node  (sub1)      {$\ldots$};
	\node  (sub2)   [right=of sub1]   {$H_{j-p}\left(\partial_-W_k\right)$};
	\node  (sub3)   [right=of sub2]    {$H_{j-p}\left(W_k\right)$};
	\node  (sub4)   [right=of sub3]    {$H_{j-p}\left(W_k,\partial_-W_k\right)$}; 
	\node  (sub5)   [right=of sub4]    {$\ldots$};

	\node     (main2)      [above=of sub2]   {$A_j$};
	\node      (main3)       [above=of sub3] {$B_j$};
	\node      (main4)       [above=of sub4] {$C_j$};
	\node  (main1)    [above=of sub1,yshift=0.35cm]   {$\ldots$};
	\node  (main5)    [above=of sub5,yshift=0.35cm]  {$\ldots$};
	
	\node at (4.5, 0.8)   (a) {$\circlearrowleft$};
	
	\draw[->] (sub1.east) -- (sub2.west) node[midway,above] {$-1$};
	\draw[->] (sub2.east) -- (sub3.west);
	\draw[->] (sub3.east) -- (sub4.west);
	\draw[->] (sub4.east) -- (sub5.west) node[midway,above] {$-1$};
	
	\draw[->] (main1.east) -- (main2.west) node[midway,above] {$-1$};
	\draw[->] (main2.east) -- (main3.west);
	\draw[->] (main3.east) -- (main4.west);
	\draw[->] (main4.east) -- (main5.west) node[midway,above] {$-1$};
	
	\draw[<-] (sub2.north) -- (main2.south) node[midway,text depth=2ex,right,rotate=-90,xshift=-0.25cm] {$\simeq$};
	\draw[<-] (sub3.north) -- (main3.south) node[right,midway,text depth=2ex,rotate=-90,xshift=-0.25cm] {$\simeq$};
	
\end{tikzpicture}
\end{figure}
where the two isomorphisms are given by the Thom isomorphism, since $ H^b_{-\infty}$ is the complement of a ball in $H$. Therefore, 
\begin{align*}
rank\big(H_j\big(W_k\times H,(\partial_-W_k)\times H\cup W_k\times& H^b_{-\infty}\big)\big)=rank\left(H_{j-p}\left(W_k,\partial_-W_k\right)\right)\\
&\geq (2k+1)^{r+1}rank(HN_{j-p}(M,\beta))
\end{align*}
We can thus divide by $(2k+1)^{r+1}$ and make $k$ go to $+\infty$, yielding
\[\# Crit_j^\beta(F)+o_{+\infty}(1)\geq rank(HN_{j-p}(M,\beta)).\]
The $o_{+\infty}(1)$ can then be taken to be arbitrarily small and, more specifically, strictly smaller than $1$ thus yielding the theorem
\end{proof}

A direct corollary of theorem \ref{thm1} is the following inequality, due to Chantraine and Murphy:
\begin{Cor}[\cite{Chantraine2016ConformalSG}]
	Let $M$ be a closed connected manifold, $\beta$ be a closed $1$-form on $M$ and, for some $m\in\mathbb{N}$, $F:M\times\mathbb{R}^m\rightarrow\mathbb{R}$ be a generating function. Calling $\beta$ the various pullbacks of $\beta$, assume $F$ is $\beta$-Morse. Then we have that:
\[\# Crit^\beta(F)\geq \sum_j rank(HN_{i}(M,\beta))\]
\end{Cor}

Therefore, from the theorem theorem \ref{thm1}, we can derive the main theorem of \cite{Chantraine2016ConformalSG} (the reader may refer themselves to the next section for the definitions of everything $\lcs$):

\begin{The}[\cite{Chantraine2016ConformalSG}]
	Let $M$ be a closed connected manifold, $\beta$ be a closed $1$-form on $M$ and call $\beta$ the various pullbacks of $\beta$. Let $\lambda$ be the canonical Liouville form on $T^*M$, and $L$ be a Lagrangian of the $\lcs$ manifold $(T^*M,\lambda,\beta)$ given by a generating function $F$ (in the $\lcs$ sense). If $F$ is $\beta$-Morse, then:
	\[\# L\cap M\geq \sum_j rank(HN_j(M,\beta)).\]
\end{The}

Note that as pointed out in \cite{Chantraine2016ConformalSG}, Hamiltonian isotopies (in the $\lcs$ sense)  of a Lagrangian with a generating function also have a generating function by Chekanov's theorem on the persistence of generating functions.

\section{(Essential) Liouville chords}\label{sec5}
\subsection{Reminders about locally conformally symplectic geometry}
In this subsection, we give the reader a laconic presentation of the notions of $\lcs$ geometry pertinent to the rest of the section.

\begin{Def}[Exact $\lcs$ manifold/structure]
Let $M$ be a manifold, $\beta\in\Omega^1(M)$ be closed and $\lambda\in\Omega^1(M)$ such that $d_\beta\lambda$ is non-degenerate. Then $(M,\lambda,\beta)$ is an exact $\lcs$ manifold.

An exact $\lcs$ structure on $M$ is an equivalence class of exact $\lcs$ manifolds given by:
\[(M,\lambda,\beta)\sim(M,e^g\lambda,\beta+dg),\; g\in C^\infty(M).\]
\end{Def}

Notice that since $d_0=d$, we recover the exact symplectic case as a particular case of this definition. In this generalization, it is also possible to define exact Lagrangians.

\begin{Def}[Exact Lagrangian]
	Let $(M,\lambda,\beta)$ be an exact $\lcs$ manifold and $L\subset M$ be a submanifold such that $i^*\lambda=d_{i^*\beta}f$ for $i$ the inclusion and $f\in C^\infty(L)$. Then $L$ is called $\beta$-exact Lagrangian. Whenever $\beta$ does not matter or is obvious, the ``$\beta$-'' will be dropped.
\end{Def}
\begin{Not}
	From now on we will omit writing pullbacks for $\beta$ whenever confusion is unlikely. For example, we will write $\beta$ for $i^*\beta$.
\end{Not}
As in symplectic geometry, the first examples of exact Lagrangians are given by the graph of functions.
\begin{Exe}
Let $M$ be a manifold, $\beta\in\Omega^1(M)$ be closed and take $f\in C^\infty(M)$. Then \[\Gamma_\beta(f):=(q,(d_\beta f)_q)\in T^*M\; :\; q\in M\]
is a $\beta$-exact Lagrangian of $(T^*M,\lambda,\beta)$, for $\lambda$ the canonical Liouville form.
\end{Exe}
\begin{Not}
	From now on $\beta$ will be a closed $1$-form on a manifold $M$ and $\lambda$ will be the canonical Liouville form of $T^*M$.
\end{Not}
This example is generalized by the notion of generating function.
\begin{Def}[Generating function]
	Let $M$ be a manifold and, for some $k\in\N$, let $F:M\times\R^k\rightarrow\R$ be a smooth map. If there is a compact $K\subset\R^k$ such that $F$ is quadratic outside of $M\times K$ (aka. is ``quadratic at infinity''), then $F$ will be called a generating function.
	
	Define \[V_F:=\Gamma_\beta(F)\cap (T^*M)\times\R^k\subset T^*(M\times\R^k)\]
	For a generic $F$, this is a submanifold, and the projection of $V_F$ on $T^*M$ is an immersed submanifold, called the (immersed) Lagrangian submanifold associated to $F$ and denoted $L_F$.
	
	The generating function $F$ will be said to be strictly positive on $L_F$ if the pullback of $\lambda$, the canonical Liouville form of $T^*M$, on $L$ is equal to $d_\beta f$ for some strictly positive map $f\in C^\infty(L_F)$.
\end{Def}
As explored in \cite{currier2025projectionexactlagrangianslocally}, Liouville chords seem to be the $\lcs$ version of the Reeb chords of contact geometry.
\begin{Def}[(essential) Liouville chord]
	For $k\in\{1,2\}$, let $L_k$ be exact Lagrangian submanifolds of $(T^*M,\lambda,\beta)$ that intersect transversely such that $i^*\lambda=d_\beta f_k$ for $i$ the inclusion and $f_k\in C^\infty(L_k,\R_+^*)$. We will assume that $\beta$ is not exact. Call $\Phi$ the flow of the canonical Liouville vector field. For any $x\in T^*M$, we have the following definitions:
	\begin{enumerate}
		\item 	A (positive) Liouville chord from $L_1$ to $L_2$ is a trajectory $\{\Phi_s(x):s\in[0,t]\}$ such that $\Phi_0(x)\in L_1$ and $\Phi_t(x)\in L_2$. We will call $t$ the length of the Liouville chord.
		\item A negative Liouville chord from $L_1$ to $L_2$ is a trajectory $\{\Phi_s(x):s\in[-t,0]\}$ such that $\Phi_0(x)\in L_1$ and $\Phi_{-t}(x)\in L_2$. We will call $-t$ the length of the Liouville chord.
		\item	A (positive) Liouville chord $\{\Phi_s(x):s\in[0,t]\}$ will be called essential if $f_2(\Phi_t(x))-e^tf_1(\Phi_0(x))\geq 0$.
		\item 	A negative Liouville chord $\{\Phi_s(x):s\in[-t,0]\}$ will be called essential if $f_2(\Phi_{-t}(x))-e^{-t}f_1(\Phi_0(x)))\leq 0$.
		\item A Liouville chord of $L_1$ is a trajectory $\{\Phi_s(x):s\in[0,t]\}$ such that $\Phi_0(x),\Phi_{t}(x)\in L_1$ for some $t>0$. We will call $t$ the length of the Liouville chord.
		\item 	A Liouville chord $\{\Phi_s(x):s\in[0,t]\}$ of $L_1$ will be called essential if $f_1(\Phi_t(x))-e^tf_1(\Phi_0(x))\geq 0$.
	\end{enumerate}
\end{Def}
\subsection{Proof of proposition \ref{prop}}
In this subsection, we will prove proposition \ref{prop}. Let us state the first and most obvious lemma concerning the $t$-difference function defined in definition \ref{t-diff}:
\begin{Lem}\label{lemma1}
Let $F_1:M\times\mathbb{R}^{k_1}\rightarrow\R^*$ and $F_2:M\times\mathbb{R}^{k_2}\rightarrow\R^*$ be generating functions that are strictly positive on the Lagrangians they define.\begin{enumerate}
\item Liouville chords of length $t$ from $L_{F_1}$ to $L_{F_2}$ are given by $Crit^\beta(\Delta^t_{F_1,F_2})$;
\item  essential positive Liouville chords of length $t$ from $L_{F_1}$ to $L_{F_2}$ are given by: \[Crit^\beta(\Delta^t_{F_1,F_2})\cap \{\Delta^t_{F_1,F_2}\geq 0\};\]
\item essential negative Liouville chords of length $t$ from $L_{F_1}$ to $L_{F_2}$ are given by: \[Crit^\beta(\Delta^t_{F_1,F_2})\cap \{\Delta^t_{F_1,F_2}\leq 0\};\]
\item essential Liouville chords of length $t$ of $L_{F_1}$ are given by: \[Crit^\beta(\Delta^t_{F_1,F_1})\cap \{\Delta^t_{F_1,F_1}\geq 0\}.\]
\end{enumerate}
\end{Lem}
\begin{proof}
	Let $Z_\lambda$ be the Liouville vector field of the canonical Liouville form $\lambda$ on $T^*M$, and call $\Phi$ its flow. We will call $\beta$ the various pullbacks of $\beta$.
	
	Note that for, say, $F_1$, we have that $d_\beta F_1= D_xF_1+D_\xi F_1-F\beta$ where $D_x$ denotes the differential along $M$ and $D_\xi$ is the differential along $\mathbb{R}^{k_1}$.
	
	There is a Liouville chord from $L_{F_1}$ to $L_{F_2}$ in $T^*_xM$ for some $x\in M$ if and only if the following conditions are met:\begin{enumerate}
	\item there is some $\xi_1\in\mathbb{R}^{k_1}$ and $\xi_2\in\mathbb{R}^{k_2}$ such that $(D_\xi F_1)_{(x,\xi_1)}=0$ and $(D_\xi F_2)_{(x,\xi_2)}=0$;
	\item viewing $(d_\beta F_1)_{(x,\xi_1)}$ and $(d_\beta F_2)_{(x,\xi_2)}$ in $T^*_xM$, we have \[\phi_t(x,(d_\beta F_1)_{(x,\xi_1)})=(x,(d_\beta F_2)_{(x,\xi_2)}).\]
	\end{enumerate} 
	This last condition implies that \[(x,e^t\times(d_\beta F_1)_{(x,\xi_1)})=(x,(d_\beta F_2)_{(x,\xi_2)}),\] and therefore $d_\beta \Delta^t_{F_1,F_2}=0$.

The rest of the assertions of the lemma follow by applying the definitions of the various types of Liouville chords.
\end{proof}
Take also note of the following lemma:
\begin{Lem}\label{lemma2}
Let $\pi:\tilde{M}_\beta\times\R^{k_1}\times\R^{k_2}\rightarrow M\times\R^{k_1}\times\R^{k_2}$ be the canonical projection, $i$ be the inclusion $\{\Delta^t_{F_1,F_2}\leq 0\}\subset M\times\R^{k_1}\times\R^{k_2}$ and $\mathcal{H}$ be the group of deck transformations of $\tilde{M}_\beta\times\R^{k_1}\times\R^{k_2}$. Then \begin{align*}
rank&_{Nov(i^*\beta)}\bigg(HN_i(\{\Delta^t_{F_1,F_2}\leq 0\},i^*\beta)\bigg)\\
&\leq rank_{\mathbb{R}(\mathcal{H})}\bigg(H_i\bigg(C_*\big(\pi^{-1}(\{\Delta^t_{F_1,F_2}\leq 0\}),\mathbb{R}\big)\otimes_{\mathbb{R}[\mathcal{H}]}\mathbb{R}(\mathcal{H})\bigg)\bigg).
\end{align*}
\end{Lem}

\begin{proof}
Note that $\pi^{-1}(\{\Delta^t_{F_1,F_2}\leq 0\})$ is an integral cover for the pullback of $\beta$ to $\{\Delta^t_{F_1,F_2}\leq 0\}$ and the group of deck transformations of $\tilde{M}_\beta\times\R^{k_1}\times\R^{k_2}$ is a subgroup of the group of deck transformations of $\pi^{-1}(\{\Delta^t_{F_1,F_2}\leq 0\})$. Among other things, this implies that given the inclusion $i: \{\Delta^t_{F_1,F_2}\leq 0\}\rightarrow M\times\R^{k_1}\times\R^{k_2}$, we have $Nov(\beta)\subset Nov(i^*\beta)$, which means that $Nov(i^*\beta)$ is a flat $Nov(\beta)$-module. For clarity's sake, we will forgo writing the tensorization of the Novikov ring by $\mathbb{R}$, but from now on, for this proof, all the various Novikov rings will be tensorized with $\mathbb{R}$. Writing $\mathcal{H}'$ for the group of deck transformations of $\pi^{-1}(\{\Delta^t_{F_1,F_2}\leq 0\})$, we have: \[HN_i\big(\{\Delta^t_{F_1,F_2}\leq 0\}\big)\otimes\mathbb{R}\simeq H_i\big(C_*\big(\pi^{-1}(\{\Delta^t_{F_1,F_2}\leq 0\}),\mathbb{R}\big)\otimes_{\mathbb{R}[\mathcal{H}']}Nov(i^*\beta)\big).\]
Therefore, 
\begin{align*}
	&rank_{Nov(i^*\beta)}\bigg(HN_i\big(\{\Delta^t_{F_1,F_2}\leq 0\}\big)\otimes\mathbb{R}\bigg)\\
	& =rank_{\mathbb{R}(\mathcal{H}')}\bigg(H_i\bigg(C_*\big(\pi^{-1}(\{\Delta^t_{F_1,F_2}\leq 0\}),\mathbb{R}\big)\otimes_{\mathbb{R}[\mathcal{H}]}\mathbb{R}(\mathcal{H})\otimes_{\mathbb{R}[\mathcal{H}']}\mathbb{R}(\mathcal{H}')\bigg)\bigg)\\
	& =rank_{\mathbb{R}(\mathcal{H}')}\bigg(H_i\bigg(C_*\big(\pi^{-1}(\{\Delta^t_{F_1,F_2}\leq 0\}),\mathbb{R}\big)\otimes_{\mathbb{R}[\mathcal{H}]}\mathbb{R}(\mathcal{H})\bigg)\otimes_{\mathbb{R}[\mathcal{H}']}\mathbb{R}(\mathcal{H}')\bigg)\\
	& \leq rank_{\mathbb{R}(\mathcal{H})}\bigg(H_i\bigg(C_*(\pi^{-1}(\{\Delta^t_{F_1,F_2}\leq 0\}),\mathbb{R})\otimes_{\mathbb{R}[\mathcal{H}]}\mathbb{R}(\mathcal{H})\bigg)\bigg).
\end{align*}
\end{proof}

We can now use both the two previous lemmas to prove proposition \ref{prop}.

\begin{proof}[proof of proposition \ref{prop}]
The first assertion in this proposition is a direct consequence of theorem \ref{thm1}, as one needs only to sum the number of the critical points of every index. For the other parts of the proposition, it is then sufficient to check that the arguments subsection \ref{rank} carry through.\newline 

 First, we can reduce the proof to that for the negative Liouville chords. Indeed, we have that $\{\Delta^t_{F_1,F_2}\geq 0\}=\{\Delta^{-t}_{F_2,F_1}\leq 0\}$ and $Crit^\beta(\Delta^t_{F_1,F_2})=Crit^\beta(\Delta^{-t}_{F_2,F_1})$. Moreover, an essential positive Liouville chord from $L_{F_1}$ to $L_{F_2}$ is an essential negative Liouville chord from $L_{F_2}$ to $L_{F_1}$.
 
 Second, note that for a generic choice of $\beta$ in the same cohomology class, $\pi^{-1}(\{\Delta^t_{F_1,F_2}= 0\})$ intersects $\partial\overline{\tilde{V}_\beta}$ transversely. This implies that we can find a CW decomposition of $\{\Delta^{t}_{F_1,F_2}\leq 0\}$ such that for any deck transformation $\alpha\in\mathcal{H}$ and any cell $\sigma$, $\alpha(\sigma)$ is also a cell. Finally, note that since $\Delta^t_{F_1,F_2}$ is quadratic at infinity, there is some small enough $b$ such that $\{b\leq \Delta^t_{F_1,F_2}\leq 0\}$ is a deformation retract of $\{\Delta^t_{F_1,F_2}\leq 0\}$.
 
 Finally, note that \[\pi^{-1}(\{\Delta^t_{F_1,F_2}= 0\})=\{e^{-g}\Delta^t_{F_1,F_2}\circ\pi= 0\}\] for $g$ a primitive of $\beta$ on $\tilde{M}_\beta$. 
 
 Therefore, we can apply the arguments of the subsection \ref{rank} to \[E_k=W_k\times\R^{k_1}\times\R^{k_2}\cap \pi^{-1}(\{\Delta^t_{F_1,F_2}\leq 0\}),\] with \[\partial_-E_k=\partial_-W_k\times\R^{k_1}\times\R^{k_2}\cap \pi^{-1}(\{\Delta^t_{F_1,F_2}\leq 0\}),\]
 and the extension defined in subsection \ref{relevéfonction} can then be applied to $e^{-g}\Delta^t_{F_1,F_2}\circ\pi$.
  Note that extending $\Delta^t_{F_1,F_2}$ does not create new critical points along \[W_k\times\R^{k_1}\times\R^{k_2}\cap \partial\pi^{-1}(\{\Delta^t_{F_1,F_2}\leq 0\}),\] but may create critical points along \[\partial W_k\times\R^{k_1}\times\R^{k_2}\cap \pi^{-1}(\{\Delta^t_{F_1,F_2}\leq 0\})\] proportionally to the number of deck transformations in $\partial_+\mathcal{H}_k\cup\partial_-\mathcal{H}_k$. Indeed, it doesn't create new critical points along $\pi^{-1}(\{\Delta^t_{F_1,F_2}= 0\})$ since over this set $d(e^{-g}\Delta^t_{F_1,F_2}\circ\pi)=e^{-g}d(\Delta^t_{F_1,F_2}\circ\pi)$, which exits $\pi^{-1}(\{\Delta^t_{F_1,F_2}\leq 0\})$ along $\pi^{-1}(\{\Delta^t_{F_1,F_2}= 0\})$. Note that this implies that $W_k\times\R^{k_1}\times\R^{k_2}\cap \partial\pi^{-1}(\{\Delta^t_{F_1,F_2}\leq 0\})$ is part of the positive boundary.
  
  This allows us to apply the proof strategy of theorem \ref{thm1} (in section \ref{sec4}) to conclude that \[\sum_i Crit_i^\beta(\Delta^t_{F_1,F_2})\geq\sum_i rank_{\mathbb{R}(\mathcal{H})}\bigg(H_i\bigg(C_*\big(\pi^{-1}(\{\Delta^t_{F_1,F_2}\leq 0\}),\mathbb{R}\big)\otimes_{\mathbb{R}[\mathcal{H}]}\mathbb{R}(\mathcal{H})\bigg)\bigg).\]
 
  Together with lemma \ref{lemma2}, we get that \[\sum_i Crit_i^\beta(\Delta^t_{F_1,F_2})\geq\sum_i rank_{Nov(i^*\beta)}\bigg(HN_i(\{\Delta^t_{F_1,F_2}\leq 0\},i^*\beta)\bigg).\]
  And, finally, we conclude using lemma \ref{lemma1} to get the result.
\end{proof}

\bibliographystyle{plain}
\bibliography{./biblio.bib}
\end{document}